\documentclass[11pt]{amsart}\usepackage{amsmath,amssymb,amsthm,graphicx}

\usepackage{pgf}
\usepackage{tikz}
\usepackage{tikz-cd}

\usetikzlibrary{matrix,arrows,backgrounds,shapes.misc,shapes.geometric,patterns,calc,positioning}
\usetikzlibrary{calc,shapes}
\usetikzlibrary{decorations.pathmorphing}

\usepackage{epsfig}  

\usepackage{amsfonts}

\usepackage{multicol}
\usepackage[all]{xy}
\usepackage{xcolor}
\usepackage{color}
\usepackage{float}

\usepackage{graphics}

\usepackage[left=2.6cm,right=2.6cm,top=2.7cm,bottom=2.8cm]{geometry}

\newcommand{\chara}{\mathrm{char}\hspace{1pt}}

\newcommand{\ttop}{\mathrm{top}\hspace{1pt}}

\newcommand{\CMP}{{\textup{CM\hspace{1pt}}}}
\newcommand{\uCMP}{{\underline{\textup{CM}}\hspace{1pt}}}

\newcommand{\Hom}{{\textup{Hom}}}
\newcommand{\End}{{\textup{End}}}
\newcommand{\Ext}{{\textup{Ext}}}

\renewcommand{\mod}{{\textup{mod}\hspace{1pt}}}

\newcommand{\rad}{{\textup{rad}\hspace{1pt}}}

\newcommand{\add}{{\textup{add}\hspace{1pt}}}

\newcommand{\injdim}{{\textup{inj.dim}\hspace{1pt}}}
\newcommand{\projdim}{{\textup{proj.dim}\hspace{1pt}}}

\newcommand{\raw}{\rightarrow}

\newcommand{\Aa}{\mathbb{A}}

\newcommand{\tAa}{\tilde{\mathbb{A}}}

\newcommand{\Bb}{\mathcal{B}}
\newcommand{\Cc}{\mathcal{C}}
\newcommand{\Bp}{\mathcal{B}_p}
\newcommand{\Bq}{\mathcal{B}_q}

\newcommand{\Db}{\mathcal{D}^b}

\newcommand{\xra}{\xrightarrow}
\newcommand{\ra}{\rightarrow}

\newcommand{\Tt}{\mathcal{T}}

\newcommand{\Pp}{\mathcal{P}}

\newtheorem{teo}{Theorem}[section]
\newtheorem{lema}[teo]{Lemma}
\newtheorem{teorema}[teo]{Theorem}

\newtheorem{corolario}[teo]{Corollary}

\newtheorem{proposicion}[teo]{Proposition}

\newtheorem{ejemplo}[teo]{Example}
\newtheorem{remark}[teo]{Remark}

\newtheorem{definicion}[teo]{Definition}

\newtheorem{teorema*}{Theorem}
\newtheorem{corolario*}{Corollary}

\usepackage{color}

\AtEndDocument{\bigskip{\footnotesize%

  \textsc{Departamento de Matem\'atica, Universidad Nacional de Mar del Plata, Dean Funes 3350, Mar del Plata, Argentina} \par
  \textit{E-mail address} \texttt{elsener@mdp.edu.ar} \par
}}

\begin{document}\date{\today. Key words: 2-Calabi-Yau tilted algebras, Jacobian algebras, Gentle algebras.}

\title{Gentle $m$-Calabi-Yau tilted algebras}
\author{Ana Garcia Elsener}

\thanks{Supported by CONICET and PICT 2013-0799 ANPCyT. The author would like to thank: Ralf Schiffler and Pamela Su\'arez for reading the first version of this note, Juan Ver\'on for helping with the graphics, and the referee for pointing out references and previous mistakes.}

\begin{abstract} 

We prove that all gentle 2-Calabi-Yau tilted algebras (over an algebraically closed field $k$, $\chara k \neq 3$) are Jacobian, moreover their bound quiver can be obtained via block decomposition. For two related families, the $m$-cluster-tilted algebras of type $\Aa$ and $\tAa$,  we prove that a module $M$ is stable Cohen-Macaulay if and only if $\Omega^{m+1} \tau M \simeq M$. 
\end{abstract}

\maketitle

\section*{Introduction}

Gentle algebras are a class of finite dimensional algebras whose module (and derived) category is well understood. These algebras have good properties, they are Gorenstein \cite{GR}, tame, and their module category is described via strings and parametrized bands \cite{ButRi}.
On the other hand, 2-Calabi-Yau (2-CY for short) tilted algebras are generalization of the concept of cluster-tilted
algebras. A cluster-tilted algebra is the endomorphism algebra of a cluster-tilting object in the cluster category
of a hereditary algebra, 2-CY tilted algebras are obtained by replacing the cluster category by a
2-CY triangulated category. These algebras are Goresntein of dimension at most one \cite{KR}. Cluster categories and cluster-tilted algebras were introduced in \cite{BMRRT,CCS,BMR}. Jacobian algebras were defined in \cite{DWZ}, these algebras are defined by a quiver with potential $(Q,W)$.
In \cite{Am} were introduced 2-CY categories $\Cc_{(Q,W)}$ associated to quivers with potential, in such way that Jacobian algebras are obtained as 2-CY tilted algebras arising from $\Cc_{(Q,W)}$.

A well known class of gentle 2-CY tilted algebras are the Jacobian algebras arising from unpunctured surfaces defined in \cite{ABCP}. This family includes the cluster-tilted algebras of type $\Aa$ and $\tAa$. A related family of gentle algebras, that also have a geometric realization via unpunctured surfaces, are the $m$-cluster-tilted algebras of types $\Aa$ and $\tAa$ \cite{T,Ba,Gub}.

In section 2, we characterize the gentle 2-CY tilted $k$-algebras, in the case $\chara k \neq 3$ (See Remark \ref{remark char 3}). By \cite{GS}, the modules in the singularity category of a $2$-CY tilted algebra are exactly those satisfying the formula $\Omega^2\tau M \simeq M$. The corresponding modules over gentle algebras were studied in \cite{Ka}. These results motivate the following theorem.

\begin{teorema*} Let $\Lambda=kQ/I$ be a gentle algebra of Gorenstein dimension at most one and such that $\Omega^2 \tau M \simeq M$ for all $M \in \uCMP(\Lambda)$. Then, $(Q,I)$ is obtained via block decomposition, matching blocks of type I, II and loop:

\begin{center}
\includegraphics{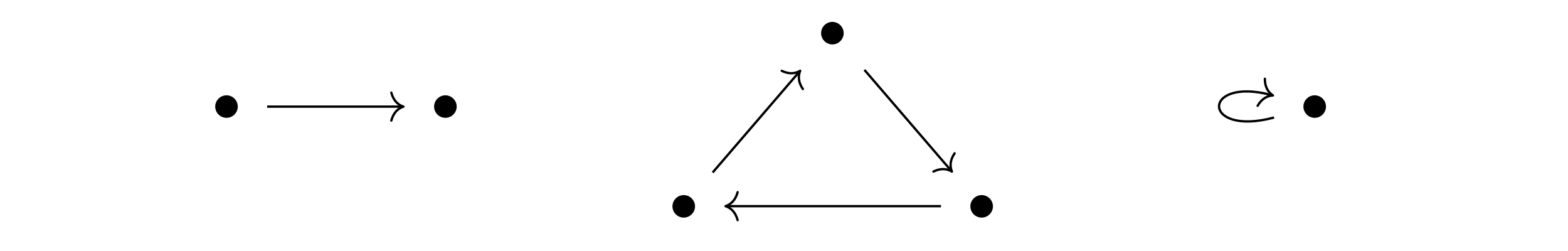}
\end{center}


\end{teorema*}

Immediately, we obtain the next result.

\begin{corolario*}
If $\Lambda = kQ/I$ is gentle 2-CY tilted and $\chara k \neq 3$, then $\Lambda$ is Jaciobian. 
\end{corolario*}

In section 3, we prove a result generalizing the formula that characterizes the modules in the singularity category of $2$-CY tilted algebras, in the context of gentle $m$-cluster tilted algebras.

\begin{teorema*}\label{teo m tipo a} Let $\Lambda$ be a $m$-cluster tilted algebra of type $\Aa$ or $\tAa$. Then,
\begin{enumerate}
\item $\Lambda$ is Gorenstein of dimension $d \leq m$.
\item $N \in \uCMP(\Lambda)$ if and only if $\Omega^{m+1} \tau N = N$.
\end{enumerate}
\end{teorema*}

The paper is organized as follows. In section 1, we recall basic facts about gentle algebras, 2-CY tilted algebras and $m$-cluster-tilted algebras. Section 2 is devoted to our characterization of gentle 2-Calabi-Yau tilted algebras. The study of the modules in singularity categories over $m$-cluster-tilted algebras is given in section 3.

\section{Preliminaries}\label{sect 1}

Throughout these notes, let $k$ be an algebraically closed field and let $Q=(Q_0,Q_1)$ be a finite quiver, where $Q_0$ is the set of vertices and $Q_1$ the set of arrows. Let $s,t \colon Q_1 \ra Q_0$ be the functions that indicate the source and the target of each arrow, respectively. We will only consider \emph{finite-dimensional basic $k$-algebras}. Every finite- dimensional basic $k$-algebra is isomorphic to a quotient $kQ/I$, where $I$ is an admissible ideal. The pair $(Q,I)$ is called a \emph{bound quiver}. 
For more details, see \cite[Chapter III]{ASS}.

\subsection{Gentle algebras}

We recall the definition of gentle algebra and results due to Geiss and Reiten \cite{GR}, and Kalck \cite{Ka}. 

\begin{definicion} A $k$-algebra $\Lambda = k Q/I$  is gentle if 
\begin{enumerate}
\item[(G1)] For each vertex $x_0 \in Q_0$ there are at most two arrows such that $x_0$ is their source, and at most two arrows such that $x_0$ is their target.
\item[(G2)] The ideal $I$ is generated by paths  of length $2$. 
\item[(G3)] For each arrow $\beta$ there is at most one arrow $\alpha$ and at most one arrow
$\gamma$ such that $\alpha \beta \in I$ and $\beta \gamma \in I$.
\item[(G4)] For each arrow $\beta$ there is at most one arrow $\alpha$ and at most one arrow
$\gamma$ such that $\alpha \beta \notin I$ and $\beta \gamma \notin I$.
\end{enumerate} 
\end{definicion}

We will often refer to the generators in $I$ as \emph{zero-relations}.

\begin{definicion} Let $\Lambda = kQ/I$ be a gentle algebra.

\begin{enumerate}

\item[(a)] A cycle $x_1 \xra{\alpha_1}  \cdots \ra x_{n}\xra{\alpha_n} x_1 $ is saturated if $\alpha_i \alpha_{i+1} \in I$, for $i$ an integer modulo $n$. In particular, a saturated loop is an arrow $\delta$ such that $s(\delta)=t(\delta)$ and $\delta^2 \in I$. 

\item[(b)] An arrow $\beta $ is gentle if there is no other arrow $\alpha$ such that $\alpha \beta \in I$.

\item[(c)] A path $\alpha_1 \ldots \alpha_n$ is formed by consecutive relations if $\alpha_i \alpha_{i+1} \in I$ for $1 \leq i < n$.

\item[(d)] A path $\alpha_1 \ldots \alpha_n$ is critical if it is formed by consecutive relations and $\alpha_1 $ is a gentle arrow.

\end{enumerate}

\end{definicion}

When there is no gentle arrow, we set $n(\Lambda)=0$. When there is a gentle arrow, let $n(\Lambda)$ be the maximal length computed over all critical paths. This number is bounded, since $Q$ is finite.  

Let $\Omega$ be the usual syzygy operator, $\tau$ the Auslander-Reiten (AR) translation, and $D = \Hom_k ( - ,k)$.

\begin{definicion} A $k$-algebra $\Lambda$ is Gorenstein if $\injdim \Lambda = \projdim D(\Lambda^{op})=d$ for some non-negative integer $d$. In this case we say that $\Lambda$ is Gorenstein of dimension $d$.  

\end{definicion}

\begin{teorema} \label{lema GR} \cite{GR} Let  $ \Lambda= kQ/I$ be a gentle algebra with $n(\Lambda)$ the maximum
length of critical paths. Then $\injdim \Lambda = n(\Lambda) =
\projdim D(\Lambda^{op})$ if $n(\Lambda) > 0$, and $\injdim \Lambda = \projdim D(\Lambda^{op}) \leq 1$ if $n(\Lambda) = 0$. In particular, $\Lambda$ is Gorenstein.

\end{teorema}

An algebra $\Lambda = kQ/I$
where $I$ is generated by paths and $(Q, I)$ satisfies the two conditions (G1)
and (G4) is called a \emph{string algebra}, thus every gentle algebra is
a string algebra.  A \emph{string} in $\Lambda$ is by definition a reduced walk $w$ in $Q$ avoiding the
zero-relations, thus $w$ is a sequence $x_1 \overset{\alpha_1}{\longleftrightarrow} x_2 \overset{\alpha_2}{\longleftrightarrow} \cdots \overset{\alpha_n}{\longleftrightarrow} x_{n+1}$ where the $x_i$ are vertices of $Q$ and each $\alpha_i$ is an arrow between the vertices $x_i$ and $x_{i+1}$ in either direction such that there is no $\overset{\beta}{\longrightarrow} \overset{\beta}{\longleftarrow}$, and no $\overset{\beta_1}{\longleftarrow} \cdots \overset{\beta_t}{\longleftarrow}$ or $ \overset{\beta_1}{\longrightarrow} \cdots \overset{\beta_t}{\longrightarrow}$ with $\beta_1 \ldots \beta_t \in I$. If the first and the last vertex of $w$ coincide, then the string is cyclic. A \emph{band} is a cyclic string $b$ such that each power $b^n$ is a cyclic string but $b$ is not a power of some string. The classification of indecomposable modules over a string
algebra $\Lambda = kQ/I$ is given by Butler and Ringel in terms of strings and bands in $(Q,I)$. Each string $w$ defines an indecomposable module $M(w)$, called a \emph{string module}, and each band $b$ defines a family of indecomposable modules $M(b,\lambda,n)$, called \emph{band modules}, with parameters $\lambda \in k$ and $n \in \mathbb{N}$. We refer to \cite{ButRi} for the definition of string and band modules. 

\smallskip

Consider the subcategory of maximal Cohen-Macaulay modules (also called Gorenstein projective modules) defined by $\CMP(\Lambda)=\{ M \colon \Ext^i_\Lambda (M, \Lambda)=0 , \ for \ all \ i>0 \}.$ 
The stable category $\uCMP(\Lambda)=\CMP(\Lambda) / (\Pp)$, where $(\Pp)$ denotes the ideal of morphisms factoring through a projective $\Lambda$-module, is the \emph{singularity category} of $\Lambda$. 
Let $x_1 \xra{\alpha_1} \cdots \xra{\alpha_{n-1}} x_n \xra{\alpha_n} x_1$ be a saturated cycle, the indecomposable proyective module $P(x_i)$ and the indecomposable injective module $I(x_i)$ are string modules given by $P(x_i)=M(u_i^{-1} \alpha_i u_{i+1})$ and $I(x_i)= M(v_{i-1} \alpha_{i-1} v_{i}^{-1})$ (see Figure \ref{kalck1}).

\begin{figure}[h!]


\begin{center}
\begin{tikzcd}[column sep=small]
\ \arrow[d,rightsquigarrow,"v_{i}"]& & \  \arrow[d,rightsquigarrow,"v_{i+1}"] & & \ \arrow[d,rightsquigarrow,"v_{i+2}"] \\ x_i \arrow[rr,"\alpha_i"] \arrow[d,rightsquigarrow,"u_i"] & & x_{i+1} \arrow[rr,"\alpha_{i+1}"] \arrow[d,rightsquigarrow,"u_{i+1}"]& & x_{i+2}\arrow[d,rightsquigarrow,"u_{i+2}"] \\
\ & & \ & & \
\end{tikzcd}
\end{center}
\caption{\sf Local situation for a saturated cycle. The path $u_i$ is the maximal path starting at the vertex $x_i$ and the path $v_i$ is the maximal path ending at $x_i$.}\label{kalck1}
\end{figure}

\begin{remark}\label{rema gorenstein} For a Gorenstein algebra $\Lambda$ of dimension $d$, a $\Lambda$-module $M$ is Cohen-Macaulay if and only if $M$ is a $d$-th syzygy, see \cite[Proposition 6.20]{Be00}. In this case each $\Lambda$-module either has infinite projective dimension or has projective dimension at most $d$.
\end{remark}

We are interested in computing projective resolutions and AR translations over the modules in $\uCMP(\Lambda)$. For that reason, we need the next result.

\begin{teorema}\cite[Theorem 2.5]{Ka}\label{teorema de kalck} Let $\Lambda = k Q/I$ be a gentle algebra. Let $x_1 \xra{\alpha_1} \cdots \xra{\alpha_{n-1}} x_n \xra{\alpha_n} x_1$ be a saturated cycle. The string module $M(u_i)$, where $u_i$ is the string starting at $x_i$ as in Figure \ref{kalck1}, is Cohen-Macaulay. Moreover, all indecomposable modules in $\uCMP(\Lambda)$ are obtained in such manner. 
\end{teorema}

\subsection{2-CY tilted algebras and Jacobian algebras} A triangulated $k$-category $\Cc$, $\Hom$-finite with split idempotents, is  \emph{$d$-Calabi-Yau} ($d$-CY for short) if there is a bifunctorial isomorphism
\begin{equation*}
\Hom_\Cc (X,Y)\simeq D\Hom_\Cc(Y,X[d]), \ for \ all \ X,Y\in   \Cc.
\end{equation*}
Let $\Cc$ be a $2$-CY category, an object $T$ is \emph{cluster-tilting} if it is basic and 
\[\add T = \{ X\in \Cc \colon  \Hom_\Cc(X,T[1])=0\}. \] 
\begin{definicion} The endomorphism algebra of a cluster-tilting object, $\End_\Cc(T)$, is called a $2$-CY tilted algebra. 
\end{definicion}
Examples of $2$-CY tilted algebras are the \emph{cluster-tilted} algebras defined in \cite{BMR}. We will need a result due to Keller and Reiten.

\begin{proposicion}\label{prop KR} \cite{KR} Let $\Lambda$ be a $2$-CY tilted algebra, $\Lambda$ is Gorenstein of dimension less than or equal to one. 
\end{proposicion}

We will also need the next result.

\begin{teorema}\label{teorema ralf} \cite{GS} Let $\Lambda$ be a 2-CY tilted algebra. Then, $M \in \uCMP (\Lambda)$ if and only if $\Omega^2 \tau M \simeq M$.
\end{teorema}

\emph{Quivers with potential} were introduced in \cite{DWZ}. A potential $W$ is a (possibly infinite) linear combination of cycles in
$Q$, up to cyclic equivalence. Given an arrow $\alpha$ and a cycle $\alpha_1 \ldots \alpha_l$, the cyclic derivative $\partial_\alpha$ is defined by \[\partial_\alpha (\alpha_1 \ldots \alpha_l)= \sum_{k+1}^{l} \delta_{\alpha \alpha_k} \alpha_{k+1} \ldots \alpha_l \alpha_1 \ldots \alpha_{k-1}, \]
where $\delta_{\alpha \alpha_k}$ is the Kronecker delta, and $\partial_\alpha$ extends by linearity. Notice that the cycle $\alpha_1 \ldots \alpha_l$ may have repetitions. Let $R \langle\langle Q\rangle\rangle$ be the complete path algebra consisting of all (possibly
infinite) linear combinations of paths in $Q$. Let $(Q,W)$ be a quiver with potential, the \emph{Jacobian algebra} is defined to be $Jac(Q,W)= R\langle\langle Q\rangle\rangle/ \langle\partial_\alpha W, \alpha \in Q_1\rangle$. 

\smallskip

Amiot \cite[Sec. 3]{Am} showed that Jacobian algebras are $2$-CY tilted constructing a 2-CY category $\Cc_{(Q,W)}$, the result just asks $Jac(Q,W)$ to be Jacobi-finite, this means that $Jac(Q,W)$ is finite-dimensional as a $k$-algebra. In \cite{Am11}, Amiot asked whether all $2$-CY tilted algebras are Jacobian algebras. In section \ref{seccion 2} we study gentle algebras and prove that the answer is affirmative when $\chara k \neq 3$.

\subsection{$m$-cluster categories and $m$-cluster tilted algebras}

The \emph{cluster category} $\Cc_Q$ associated to $Q$ was introduced in \cite{BMRRT} as the quotient category $\Db( \mod kQ)$ over the functor $F = \tau^{-1} [1]$.
The \emph{$m$-cluster category} associated to $Q$, that we denote by $\Cc^m_Q$, was defined in \cite{T} as the quotient category $\Db( \mod kQ)$ over the functor $F_m = \tau^{-1} [m]$. 
The category $\Cc^m_Q$ is triangulated. 

A basic object $T$ in $\Cc^m_Q$ is called \emph{$m$-cluster tilting} if
\begin{enumerate}
\item[$\ast$] $\Ext^{i}_{\Cc^m_Q} (T,T')=0$ for all $T,T' \in \add T$, for $i=1, \ldots, m$
\item[$\ast$] if $X \in \Cc^m_Q$ is such that $\Ext^{i}_{\Cc^m_Q} (X,T)=0$ for all $T \in \add T$, and for $i=1, \ldots, m$, then $X \in \add T$.
\end{enumerate}

\begin{remark}\label{rema10} We follow the notation in \cite{T}. It is important to recall that $\Cc^m_Q$ is an example of an $(m+1)$-CY category and the category $\add T$, were $T$ is an $m$-cluster tilting object, is an example of an $(m+1)$-cluster tilting subcategory in the sense of \cite{KR,KR2,Be15}.
\end{remark}

Any $m$-cluster-tilting object has $\vert Q_0 \vert$ summands, \cite[Theorem 2]{T}. The endomorphism algebra $\Lambda= \End_{\Cc^m_Q} (T)$ is called an $m$-\emph{cluster tilted algebra}, and it is a case of $(m+1)$-Calabi-Yau tilted algebra. If $Q$ is such that its underlying graph $\Delta_Q$ is a Dynkin or euclidean graph, we say that $\Lambda= \End_{\Cc^m_Q} (T)$ is a $m$-cluster tilted algebra of type $\Delta_Q$. When $\Delta_Q$ is of type $\Aa$ or $\tAa$, the $m$-cluster categories and their corresponding $m$-cluster tilting objects were realized geometrically and studied in \cite{Ba,Mu}, and \cite{Tol,Gub}, respectively. These geometric realizations generalize those from \cite{CCS,ABCP,BZ} in the case of the unpunctured disc and the annulus. 

\vspace{6pt}

\textit{\bf The case $\Aa$}: Let $\Pi$ be a disk with $nm+2$ marked points (or equivalently a $nm+2$-gon). A \emph{$m$-diagonal} is a diagonal dividing $\Pi$ into an $(mj+2)$-gon and an $(m(n-j) + 2)$-gon for some $ 1 \leq j \leq n-1/2 $. A $(m+2)$-angulation is a collection of non-intersecting $m$-diagonals that form a partition of $\Pi$ into $(m+2)$-gons. There are bijections:

\begin{center}
\begin{tabular}{ccc}
$m$-diagonals in $\Pi$ &$\leftrightarrow$ &indecomposable objects in $\Cc^m_{\Aa}$
\\
$(m+2)$-angulations of $\Pi$ &$\leftrightarrow$& $m$-cluster tilting objects in $\Cc^m_{\Aa}$

\end{tabular}
\end{center}

\textit{ \bf The case $\tAa$}: While in \cite{Tol,Gub} the authors use an annulus with marked points, we will use the universal cover given by the strip $\Sigma$, having copies of the $mp$, and $mq$, marked points on the boundary components that we denote $\mathcal{B}_p$, and $\mathcal{B}_q$, respectivelly. The points having the same label in a fixed boundary are considered up to equivalence given by congruence modulo $mp$ and $mq$. There are (isotopy classes of) arcs on $\Sigma$, called \emph{$m$-diagonals}. Each $m$-diagonal belongs to a family:
\begin{enumerate}
\item[$\ast$] Transjective: an arc $\alpha$ having an endpoint $x$ in $\Bp$ and the other endpoint $y$ in $\Bq$. The labels in $x$ and $y$ are congruent modulo $m$. (See Figure \ref{ejemplo2} - right)
\item[$\ast$] Regular on a $p$-tube: an arc $\alpha$ having both endpoints over $\mathcal{B}_p$, starting at $u$ going in positive direction counting $u+km+1$ steps, with $k\geq 1$.
\item[$\ast$] Regular on a $q$-tube: analogous to the previous case. 
\end{enumerate}

As in the previous case, there are bijections

\begin{center}
\begin{tabular}{ccc}
$m$-diagonals in $\Sigma$ &$\leftrightarrow$ &indecomposable rigid objects in $\Cc^m_{\tAa}$
\\
$(m+2)$-angulations of $\Sigma$ &$\leftrightarrow$& $m$-cluster tilting objects in $\Cc^m_{\tAa}$


\end{tabular}
\end{center}

Given an $(m+2)$-angulation $\Tt$ of $\Pi$ or $\Sigma$, the bound quiver $(Q_\Tt,I_\Tt)$ of the $m$-cluster tilted algebra $\Lambda_\Tt$ defined by the associated $m$-cluster tilting object is obtained form the geometric configuration, see \cite{Gub,Mu}. We recall this bound quiver construction in the following example. 

\begin{figure}[h!]
\centering
\def\svgwidth{5.3in}
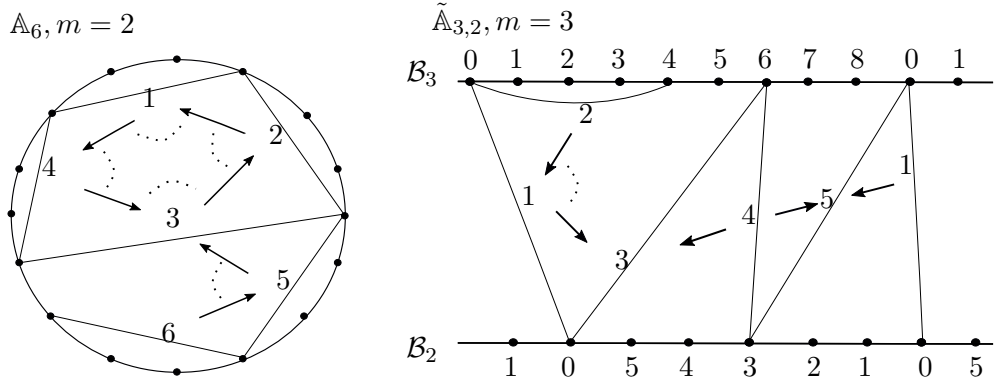
\caption{\sf $4$-angulation of $\Pi$ (left) and $5$-angulation of $\Sigma$ (right), and bound quivers defined by them. The labels at the endpoints of transjective arcs are congruent modulo $3$. The arrows defining $I_\Tt$ are connected by a dotted arc.}
\label{ejemplo2}
\end{figure}

\begin{ejemplo}\label{ejemplo1}  In Figure \ref{ejemplo2}, we show the bound quiver defined by a $4$-angulation of $\Pi$ (left) that corresponds to a $2$-cluster tilting object in $\Cc^m_{\Aa_6}$, and a $5$-angulation of $\Sigma$ (right) that corresponds to a $3$-cluster tilting object in $\Cc^m_{\tAa}$ where $p=3$ and $q=2$. In both cases, 
the vertices in $Q_\Tt$ are in one-to-one correspondence with the elements in $\Tt$. For any two vertices $i,j \in Q_\Tt$, there is an arrow $i \raw j$ when the corresponding $m$-diagonals $x_i$ and $x_j$ share a vertex, they are edges of the same $(m+2)$-gon and $x_i$ follows $x_j$ clockwise. Given consecutive arrows $i \xrightarrow{\alpha} j \xrightarrow{\beta} k$, then $\alpha \beta \in I_\Tt$ if and only if $x_i$, $x_j$ and $x_k$ are edges in the same $(m+2)$-gon. 

\end{ejemplo}

\section{Gentle $2$-CY tilted algebras}\label{seccion 2}

During this section, we work with gentle algebras $\Lambda=kQ/I$.  First we gather information about the zero-relations and saturated cycles in $(Q,I)$ in order to find all the possible configurations for a $2$-CY tilted algebra. After that, we will consider blocks, as it was done in \cite{FST}, to describe all the possible bound quivers that can define a $2$-CY tilted algebra.

\begin{lema}\label{lema loops} Let $\Lambda=kQ/I$ be a gentle algebra of Gorenstein dimension at most one, and such that $\Omega^2 \tau M \simeq M$ for all $M$ in $\uCMP(\Lambda)$. Then,
\begin{enumerate}
\item each zero-relation lies in a saturated cycle.
\item all saturated cycles have length three or are saturated loops.
\end{enumerate}

\end{lema}

\begin{proof} 
(1) Let $uv$ be a zero-relation that is not part of a saturated cycle. If $u$ is a gentle arrow, then $uv$ is a critical path of length two, so $n(\Lambda)\geq 2$ and $\Lambda$ is Gorenstein of dimension at least two. Absurd, by Proposition \ref{prop KR}. If the arrow $u$ is not gentle, the there is an arrow $u_1$ such that $u_1u$ and $uv$ are zero-relations. Since $uv$ is not part of a saturated cycle there is a maximal critical path $u_t \ldots u_1$ such that $u_t \ldots u_1 uv$ is a critical path, because $Q$ is finite and none of these arrows can be in a saturated cycle. Then we have $n(\Lambda)\geq 2$, again it is not possible by Proposition \ref{prop KR}. Thus $uv$ must lay in a saturated cycle.
\\
(2) Let $x_1 \xra{\alpha_1} \cdots \xra{\alpha_{n-1}} x_{n}\xra{\alpha_n} x_1$ be a saturated cycle. The non-zero paths starting and ending at the vertices $x_i$ determine the indecomposable projective and injective modules $P(x_i)=M(u_i^{-1} \alpha_i u_{i+1})$ and $I(x_i)= M(v_{i-1} \alpha_{i-1} v_{i}^{-1})$. See Figure \ref{kalck1}. By Theorem \ref{teorema de kalck}, $M(u_i)$ is in $\uCMP$. By Theorem \ref{teorema ralf}, we have an isomorphism $\Omega^2 \tau M(u_i) =M(u_i)$. We start computing $\tau M(u_i)$, first we need a minimal projective presentation. 
\[\xymatrix@C8pt@R8pt{&M( u_{i+1}^{-1} \alpha_{i+1} u_{i+2}) \ar[rr]^*-<1pt>{f}\ar@{->>}[rd] && M( u_i^{-1} \alpha_i u_{i+1})
\ar[rr]&& M(u_i) \ar[rr]&&0.\\ M(u_{i+2}) \ar@{^{(}->}[ru]&&M(u_{i+1}) \ar@{^{(}->}[ru]}\]
We apply the Nakayama functor $\nu= D \Hom_\Lambda (-,\Lambda)$,
\begin{displaymath}
	\xymatrix  @R=0.6cm  @C=0.6cm {
		0\ar[r] & M(v_{i+1})\ar@{^(->}[r] & M(v_i \alpha_i  v_{i+1}^{-1} )\ar[rr]^*-<1pt>{\nu f}\ar@{^(->}[rd] && M(  v_{i-1}\alpha_{i-1} v_i^{-1} )\ar@{>>}[r]&M(v_{i-1})\\
		&&& M(v_i)\ar@{>>}[ru] &&
	}
	\end{displaymath}
the kernel of $\nu f$ is $\tau M(u_i)= M(v_{i+1})$. 
\begin{figure}[h!]


\begin{center}
\begin{tikzcd}[column sep=small]
& & y_{i+1}  \arrow[d,rightsquigarrow,"v_{i+1}"]\arrow[rr,rightsquigarrow,"w_{i+1}"] & & \ \\
x_i \arrow[rr,"\alpha_i"] & & x_{i+1} \arrow[rr,"\alpha_{i+1}"] & & x_{i+2}\arrow[d,rightsquigarrow,"u_{i+2}"] \\
& & & & \
\end{tikzcd}
\end{center}
\caption{\sf Local information needed to compute the projective cover of $M(v_{i+1})$}
\end{figure}

Now we compute $\Omega M(v_{i+1})$. Let $S(y_{i+1})=\ttop M(v_{i+1})$. The projective cover of $M(v_{i+1})$ is $P(y_{i+1})=M(w_{i+1}^{-1}v_{i+1} \alpha_{i+2} u_{i+2})$. Then, $\Omega M(v_{i+1})= M(w'_{i+1}) \oplus M(u_{i+2})$, where $M(w'_{i+1})$ is the maximal sumbmodule of the uniserial $M(w_{i+1})$ and is at the same time a direct summand of $\rad P(y_{i+1})$. Therefore $\Omega^2 \tau M(u_i)= \Omega(M(w'_{i+1}) \oplus M(u_{i+2}) )= \Omega M(w'_{i+1}) \oplus \Omega M(u_{i+2})$. It is easy to see that $\Omega M(u_{i+2})=M(u_{i+3})\neq 0$. Therefore, by hypothesis, it must be $\Omega M(w'_{i+1})=0$ and $M(u_{i+3})=M(u_i)$, thus $\ttop M(u_{i+3})=S(x_{i+3})=\ttop M(u_i)=S(x_i)$, for all $i\in \{1,\ldots,n\}$. The only possible long saturated cycles satisfying this condition are of length three $x_1 \xra{\alpha_1} x_2 \xra{\alpha_2} x_3 \xra{\alpha_3} x_1$ where $x_1,x_2,x_3$ are different vertices. Indeed, if there were only two different vertices,  let 
\begin{tikzcd}
\arrow[loop left,"\alpha_1"]{l} \bullet 
\end{tikzcd}
be a saturated cycle such that $\alpha_i \alpha_{i+1}\in I$ for $i$ index modulo 3, then by the gentleness $\alpha_1^2 \notin I$ so this will define an infinite dimensional algebra. If the saturated cycle is $ \bullet\overset{\alpha_2}{\underset{\alpha_1}\rightleftarrows} \bullet $, then the condition $\ttop M(u_i)=\ttop M(u_{i+3})$ does not hold. The last possibility is considering loops. Notice that two different loops attached to a vertex 
\begin{tikzcd}
\arrow[loop left,"\delta_1"]{l} \bullet \arrow[loop right,"\delta_2"]{r}
\end{tikzcd} 
would define an infinite dimensional algebra. The last option is using a single loop 
\begin{tikzcd}
\arrow[loop left,"\delta"]{l} \bullet 
\end{tikzcd} 
such that $\delta^2 \in I$, besides the 3-cycle with three vertices, this is the only configuration allowing both conditions: $M(u_{i+3})=M(u_i)$ and $\Lambda$ is finite dimensional.\end{proof}

The algebras arising from surface triangulations $\Lambda_\Tt =k Q_\Tt/I_\Tt$ ($m=1$) were defined in \cite{ABCP}. Following \cite[Sec. 13]{FST}, the quiver $Q_\Tt$ can be constructed matching directed graphs, or \emph{blocks}, of type I (a single arrow), and type II (3-cycle).

\begin{center}
\begin{tikzcd}[column sep=small]
& \bullet  \arrow[rr,"\lambda"] & & \bullet 
\end{tikzcd}\hspace{30pt}\begin{tikzcd}[column sep=small]
& \bullet  \arrow[dr,"\alpha"] & \\
\bullet \arrow[ur,"\gamma"] & & \bullet \arrow[ll,"\beta"] 
\end{tikzcd}
\end{center}
In view of Lemma \ref{lema loops}, in the following subsection we will add a new block \begin{tikzcd}
\arrow[loop left,"\delta"]{l} \bullet 
\end{tikzcd}. We call this block \emph{type loop}.

\subsection{Block decomposition} Consider the blocks type I, II and loop mentioned above. The blocks contain also the information 
\begin{enumerate}
\item[(R1)] For a block of type II, $\alpha \beta = \beta \gamma= \gamma \alpha= 0$.
\item[(R2)] For a block of type loop, $\delta^2=0$. 
\end{enumerate}
All the vertices in the blocks are \emph{outlet vertices}. A bound quiver $(Q,I)$ is \emph{gentle-block-decomposable} if it can be obtained from a collection of
disjoint blocks by the following procedure. Take a partial matching of the combined
set of outlets.
\begin{enumerate}
\item Matching an outlet to itself or to another outlet from the same block
is not allowed.
\item  Matching two outlets corresponding to different blocks type loop is not allowed.
\end{enumerate}  

Identify (or \emph{glue}) the vertices within each pair of the matching.  Afer the gluing , having a pair of arrows
connecting the same pair of vertices but going in opposite directions is not allowed. The next is the main result of this section.

\begin{teorema}\label{Teorema} Let $\Lambda=kQ/I$ be a gentle algebra of Gorenstein dimension at most one and such that $\Omega^2 \tau M \simeq M$ for all $M \in \uCMP(\Lambda)$. Then, the bound quiver $(Q,I)$ is gentle-block-decomposable.
\end{teorema}

\begin{proof} By Lemma \ref{lema loops}, the only zero-relations allowed are $\delta^2=0$, when $\delta$ is a loop, and those in a saturated 3-cycle. All gentle bound quivers satisfying these conditions can be built matching blocks of type I, II and loop. 
\end{proof}

\begin{remark}\label{remark char 3} The block decomposition for $Q$ may include blocks of type loop, in order to interpret $(Q,I)$ as a Jacobian algebras we also need that $\chara k \neq 3$. The reason is that we use the quiver 
\begin{tikzcd}
\arrow[loop left,"\delta"]{l} \bullet 
\end{tikzcd} 
with potential $W=\delta^3$ and, in this case, the ideal for the Jacobian algebra is generated by $\partial_\delta (\delta^3) = \delta \delta + \delta \delta +\delta \delta = 3\delta^2$.
\end{remark}

\begin{corolario} Let $k$ be algebraically closed and $\chara k \neq 3$. Let $kQ/I$ be a gentle 2-CY tilted algebra, then $kQ/I$ is Jacobian.
\end{corolario}  

\begin{proof}
By Proposition \ref{prop KR} and Theorem \ref{teorema ralf}, $kQ/I$ is Gorenstein dimension at most one and $\Omega^2 \tau M \simeq M$ for all $M \in \uCMP(\Lambda)$, so we are under the hypothesis of Theorem \ref{Teorema}. Then $(Q,I)$ is gentle-block-decomposable. We can express $kQ/I$ as a Jacobian algebra $Jac(Q,W)$. The potential is the sum  
\[W=\sum_i \alpha_i \beta_i \gamma_i + \sum_j \delta_j^3, \]
where the index $i$ runs over all the $3$-cycles $\alpha_i \beta_i \gamma_i$ and the index $j$ runs over all the loops. Since $\chara k \neq 3$ the zero-relation $3\delta^2=0$ remains, and is equivalent to $\delta^2=0$.
\end{proof}

\begin{ejemplo}\label{ejemploloop} Let $Q$ be the quiver in Figure \ref{ejemploconloop}, and consider the potential $W=\delta_1^3 + \sum_{i=1}^3 \alpha_i \beta_i \gamma_i$. Then $Jac(Q,W)$ is a gentle 2-CY tilted algebra and $Q$ is a matching of two blocks of type I, three blocks of type II and a block of type loop.

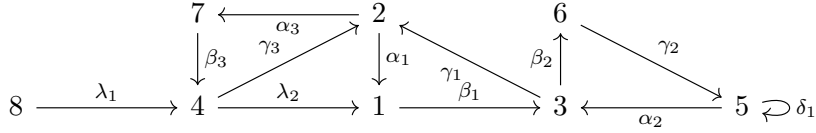
\begin{figure}[h!]


\begin{tikzcd}
&& 7 \arrow[d,"\beta_3"] 
& & 2 \arrow[ll,"\alpha_3"] \arrow[d,"\alpha_1"] && 6 \arrow[drr,"\gamma_2"] \\
8 \arrow[rr,"\lambda_1"] && 4 \arrow[urr,"\gamma_3"]\arrow[rr,"\lambda_2"]
& & 1\arrow[rr,"\beta_1"] & & 3\arrow[ull,"\gamma_1"] \arrow[u,"\beta_2"]& & 5\arrow[ll,"\alpha_2"]\arrow[loop right,"\delta_1"]{r}\end{tikzcd}
\caption{\sf Gentle bound quiver $(Q,I)$, Example \ref{ejemploloop}}
\label{ejemploconloop}
\end{figure}

\end{ejemplo}

\begin{corolario} Let $(Q,I)$ be a bound quiver such that $kQ/I$ is a gentle 2-CY tilted algebra and $Q$ has no loops. Then $kQ/I$ is a Jacobian algebra arising from an unpunctured surface, in the sense of \cite{ABCP}. 
\end{corolario}

\begin{remark}
\begin{enumerate}
\item Examples of 2-CY tilted algebras that are not Jacobian were constructed in \cite[Proposition 2]{Lad} when the field is of positive characteristic.

\item  The finite representation type 2-CY tilted algebras obtained via the block decomposition defined in this section appear in \cite[Sec. 5]{BeO}. In this work the authors use covering techniques to define mutations at loops and 2-cycles when the 2-CY category $\Cc$ is algebraic (that is, $\Cc$ is a stable category of some Frobenuis category). 
\end{enumerate}

\end{remark}

\section{Gentle $m$-cluster tilted algebras}

In this section we study gentle algebras $k Q_\Tt/I_\Tt$ arising from $(m+2)$-angulations, $m\geq 1$. One can generalize the definition of $(Q_\Tt,I_\Tt)$, given in Example \ref{ejemplo1}, to $(m+2)$-angulations of unpunctured Riemann surfaces. This was done in \cite{DR}, where the A-G invariant for gentle algebras arising from $(m+2)$-angulations was computed. 

\smallskip

The following properties were observed in \cite[Rem. 2.18]{Mu} and \cite[Sec. 7]{Gub} in the case of the disc and annulus, and they can be easily proved in the context of gentle algebras arising from $(m+2)$-angulations.

\begin{proposicion}\label{m angulaciones} Let $(Q_\Tt,I_\Tt)$ be a bound quiver arising from a $(m+2)$-angulation.
\begin{enumerate}
\item $\Lambda_\Tt = k Q_\Tt/I_\Tt$ is a gentle algebra.
\item The only possible saturated cycles in $(Q_\Tt,I_\Tt)$ are $(m+2)$-cycles.
\item There can be at most $m - 1$ consecutive zero-relations not lying in a saturated cycle.
\end{enumerate}

\end{proposicion}

Immediately, we have the following observation. 

\begin{lema}\label{lema1} Let $\Lambda_\Tt = k Q_\Tt/I_\Tt$ be an algebra arising from a $(m+2)$-angulation. Then, $\Lambda_\Tt$ is Gorenstein of dimension $d\leq m$.
\end{lema}

\begin{proof}

The case $m=1$ follows from Proposition \ref{prop KR}, and also from \cite[Lemma 2.6]{ABCP}. Let $m\geq 2$. Since $\Lambda_\Tt$ is gentle, we can apply Theorem \ref{lema GR}. First assume that there is no gentle arrow in $(Q_\Tt,I_\Tt)$, then $n(\Lambda_\Tt)=0$, so $d$ is zero or one and $d \leq m$. The statement follows.
\\
Now, assume there are gentle arrows in $(Q_\Tt,I_\Tt)$, and let $\alpha_1$ be one of them. It follows that $\alpha_1$ is not part of a saturated cycle. Let $\alpha_1 \ldots \alpha_r$ be a critical path, since $\alpha_1$ is not part of a saturated cycle, then none of the arrows $\alpha_i$ for $1 \leq i \leq r$ is part of a saturated cycle. By Proposition \ref{m angulaciones} (3), the maximal number of consecutive zero-relations outside of a saturated cycle is $m-1$. Therefore, $r \leq m$, and by Theorem \ref{lema GR}, $\Lambda_\Tt$ is Gorenstein of dimension $d \leq m$.\end{proof}

Most of the arguments in the following lemma can be found also in \cite[Section 4]{Ka}.

\begin{lema}\label{lemaK} Let $\Lambda =kQ/I$ be a gentle algebra of Gorenstein dimension $d \geq 1$. Let $x\in Q_0$, and let $N$ be an indecomposable direct summand of $\rad P(x)$. Then,
\begin{enumerate}
\item[(a)] $N \in \uCMP (\Lambda)$, or
\item[(b)] $\projdim N \leq d-1$.
\end{enumerate}

\end{lema}

\begin{proof} If $N$ is projective, we are in case (b). Let $N$ be non projective. Let  $P(x)=M(u^{-1} \alpha^{-1} \beta w)$ be the indecomposable projective and $N=M(u)$ so that $S(t(\alpha))= \ttop M(u)$. We study the cases:

\begin{enumerate}
\item[(i)] $\alpha $ is part of a saturated cycle $x_1 \ra \cdots \ra x_i \xra{\alpha} x_{i+1} \cdots \ra  x_1$.
\item[(ii)] $\alpha$ is not part of a saturated cycle.
\end{enumerate}
(i) Let $x_i\xra{\alpha} x_{i+1}$, then $M(u)$ is a direct summand of $\rad P(x_{i})$. By Theorem \ref{teorema de kalck}, $M(u)\in \uCMP (\Lambda)$.
\\
(ii) Since $N=M(u)$ is not projective, there exists an arrow $\delta_1 $ such that $\alpha \delta_1 \in I$. Also $\delta_1$ is not part of a saturated cycle, if it were the case also $\alpha$ would be part of the saturated cycle.
Let $P(t(\alpha))= M(c^{-1} \delta_1^{-1} u)$, then there is an exact sequence
\begin{equation}\label{secdelta1}
0 \ra M(c) \ra P(t(\alpha)) \ra M(u) \ra 0.
\end{equation}
If the string module $M(c)$ is not projective, then it satisfies the same conditions as $M(u)$, so we can construct a new exact sequence

\begin{equation}\label{secdelta2}
0 \ra M(c_1) \ra P(t(\delta_1)) \ra M(c) \ra 0.
\end{equation}
Recursively, we obtain a path $\alpha \delta_1 \cdots \delta_n$ such that each quadratic factor belongs to $I$.
This process has to finish after a finite number of steps, being the direct summand $M(c_n)$ of $P(t(\delta_{n-1}))$ a projective module. If there were not finite steps and $M(c_n)$ was not projective, we would find new arrows $\delta_{n+1}, \ldots$ and form a path $\alpha \delta_1 \cdots \delta_n \cdots$ such that each quadratic factor is in $I$. The quiver $Q$ is finite, so the only way to construct an infinite path $\alpha \delta_1 \cdots \delta_n \cdots$ is reaching a saturated cycle. By the gentleness, if one of the arrows $\delta_i$ is in a saturated cycle, then all $\alpha,\delta_1, \ldots , \delta_n$ are in the saturated cycle, this contradicts the condition imposed on $\alpha$. Therefore the procedure to find the short exact sequences in Equations (\ref{secdelta1}), (\ref{seccion 2}), stops. The short exact sequences are the steps needed to find a minimal projective resolution for $M(u)$, that is finite, so $\projdim M(u) < \infty$. By Remark \ref{rema gorenstein}, we have $\projdim M(u) \leq d$. Now, we can also express $M(u)$ as $M(u)=\Omega M(\beta w)$. If we had $\projdim M(u)=d$, then we would have $\projdim M(\beta w)=d+1$ and this is impossible by Remark \ref{rema gorenstein}. Thus, $\projdim M(u)\leq d-1$. \end{proof}

To complete the previous lemma, observe that if $\Lambda$ is selfinjective (that is $\Lambda$ is Gorenstein of dimension zero) then, by definition, every indecomposable module is either projective or $\uCMP$. 

\smallskip

The next theorem is the main result of this section.

\begin{teorema}\label{main} Let $\Lambda_\Tt = Q_\Tt /I_\Tt$ be an algebra arising from a $(m+2)$-angulation and let $N$ be a $\Lambda_\Tt$-module. Then, $N \in \uCMP(\Lambda_\Tt) $ if and only if $\Omega^{m+1}\tau N \simeq N$.
\end{teorema}

\begin{proof}
Let $M$ be an indecomposable module in $\uCMP(\Lambda_\Tt)$, by Theorem \ref{teorema de kalck}, $M=M(u_i)$ where $u_i$ is the maximal non-zero path starting at $x_i$.

\begin{center}

\begin{tikzcd}[column sep=small]
\ \arrow[d,rightsquigarrow,"v_{i}"]& & \  \arrow[d,rightsquigarrow,"v_{i+1}"] & & \ \arrow[d,rightsquigarrow,"v_{i+2}"] \\ x_i \arrow[rr,"\alpha_i"] \arrow[d,rightsquigarrow,"u_i"] & & x_{i+1} \arrow[rr,"\alpha_{i+1}"] \arrow[d,rightsquigarrow,"u_{i+1}"]& & x_{i+2}\arrow[d,rightsquigarrow,"u_{i+2}"] \\
\ & & \ & & \
\end{tikzcd}
\end{center}
We compute a minimal projective presentation of $M(u_i)$, as we did in Lemma \ref{lema loops}.
\[\xymatrix@C8pt@R8pt{&M( u_{i+1}^{-1} \alpha_{i+1} u_{i+2}) \ar[rr]^{p_1}\ar@{->>}[rd] && M( u_i^{-1} \alpha_i u_{i+1})
\ar[rr]&& M(u_i) \ar[rr]&&0.\\ M(u_{i+2}) \ar@{^{(}->}[ru]&&M(u_{i+1}) \ar@{^{(}->}[ru]}\]
Observe that $\Omega^t M(u_i) = M( u_{i+t})$, where $t$ is an integer considered modulo $m+2$. Applying Nakayama functor we get
 

\begin{displaymath}
	\xymatrix  @R=0.6cm  @C=0.6cm {
		0\ar[r] & M(v_{i+1})\ar@{^(->}[r] & M(v_i \alpha_i  v_{i+1}^{-1} )\ar[rr]^{\nu p_1}\ar@{^(->}[rd] && M(  v_{i-1}\alpha_{i-1} v_i^{-1} )\ar@{>>}[r]&M(v_{i-1}).\\
		&&& M(v_i)\ar@{>>}[ru] &&}
\end{displaymath}
Then, $ \tau M(u_i) = \ker \nu p_1 = M(v_{i+1})$. As in Lemma \ref{lema loops}, let $S(y_{i+1})= \ttop M(v_{i+1})$. Let $P(y_{i+1})= M(w^{-1}_{i+1} v_{i+1} \alpha_{i+1} u_{i+2})$ be the projective cover of $M(v_{i+1})$. 

\begin{center}

\begin{tikzcd}[column sep=small]
& & y_{i+1}  \arrow[d,rightsquigarrow,"v_{i+1}"]\arrow[rr,rightsquigarrow,"w_{i+1}"] & & \ \\
x_i \arrow[rr,"\alpha_i"] & & x_{i+1} \arrow[rr,"\alpha_{i+1}"] & & x_{i+2}\arrow[d,rightsquigarrow,"u_{i+2}"] \\
& & & & \
\end{tikzcd}
\end{center}
Therefore, $\Omega M(v_{i+1}) = M(w'_{i+1}) \oplus M(u_{i+2})$, where $M(w'_{i+1})$ is the maximal submodule of $M(w_ {i+1})$. The syzygy functor is additive, then
\begin{equation*}
\Omega^{m+1} \tau M(u_i) = \Omega^{m+1} M (v_{i+1}) = \Omega^m M(w'_{i+1}) \oplus \Omega^m M(u_{i+2}).
\end{equation*}
Since $\Omega^t M(u_i) = M( u_{i+t}) $, we have
\begin{equation*}
\Omega^m M(w'_{i+1}) \oplus \Omega^m M(u_{i+2}) = \Omega^m M(w'_{i+1}) \oplus M(u_i).
\end{equation*}
Now, we only need to prove that $\Omega^m M(w'_{i+1})=0$. Observe that $ M(w'_{i+1})$ is a direct summand of $\rad P(y_{i+1})$. 

\smallskip

We know, by Lemma \ref{lema1} that $\Lambda_\Tt$ is Gorenstein of dimension $d \leq m$. By Lemma \ref{lemaK} one of the following holds: 
\begin{enumerate}
\item $\projdim M(w'_{i+1}) \leq m-1$, or
\item $M(w'_{i+1}) \in \uCMP(\Lambda_\Tt)$. 
\end{enumerate}

If (1) holds, then $\Omega^m M(w'_{i+1})=0$ and we are done.

\smallskip

We assume (2) holds, so $M(w'_{i+1}) \in \uCMP(\Lambda_\Tt)$ and prove that this leads to a contradiction.  Let $z_{i+1}$ be the vertex such that $ \ttop M( w'_{i+1}) =S (z_{i+1})$. By the description in Theorem \ref{teorema de kalck}, the vertex $z_{i+1}$ is a target of an arrow $\gamma$ in a saturated $(m+2)$-cycle and $\gamma w'_{i+1}\neq 0$. 

\smallskip

(2a) If the arrow $\gamma$ is $ y_{i+1} \xra{\gamma} z_{i+1}$, see the figure bellow (left), then there is an arrow $a_j$ in the saturated $(m+2)$-cycle, such that  $a_j \gamma \in I_{\Tt}$. Then, $a_j v_{i+1} \neq 0$ and this contradicts that $I(x_{i+1}) = M( v_i \alpha_i  v_{i+1}^{-1} )$ is the indecomposable injective associated to $x_i$. Absurd.

\begin{center}

\begin{tikzcd}[column sep=small]
\ \arrow[d,"a_{j}"]& & \   & & \  \\ y_{i+1} \arrow[rr,"\gamma"] \arrow[d,rightsquigarrow,"v_{i+1}"] & & z_{i+1} \arrow[d,rightsquigarrow,"w'_{i+1}"] \arrow[rr,"a_{j+2}"]& & \  \\
\ & & \ & & \
\end{tikzcd}
\ \ \ \
\begin{tikzcd}[column sep=small]
\ \arrow[d,dashrightarrow,"a_j"] & & \ \arrow[d,"\gamma"]  & & \  \\ y_{i+1} \arrow[rr,"a_{j+1}"] \arrow[d,rightsquigarrow,"v_{i+1}"] & & z_{i+1} \arrow[rr,"b_{j+2}"] \arrow[d,rightsquigarrow,"w'_{i+1}"] & & \  \\
\ & & \ & & \
\end{tikzcd}
\end{center}

(2b) If the arrow $\gamma$ in a saturated cycle is such that $s(\gamma)\neq y_{i+1}$, see figure above (right), there is an arrow $b_{j+2}$ following the saturated cycle such that $\gamma b_{j+2} \in I_\Tt$. Thus, we have $\gamma w'_{i+1} \neq 0$ and by gentleness, $a_ {j+1}b_{j+2}\notin I_\Tt$.  But recall that $w'_{i+1}$ is a submodule of $\rad P(y_{i+1})$, so $b_{j+2}$ has to be the first arrow in the string $w'_{i+1}$ such that $M(w'_{i+1}) \in\uCMP (\Lambda_\Tt)$. Absurd.


\smallskip

Thus, $M(w'_{i+1}) \notin \uCMP (\Lambda_\Tt)$ and case (1) is the only possibility so $\Omega^{m+1} \tau M(u_i) = M(u_i)$. As $\Omega$ and $\tau$ are additive functors, if $N\in \uCMP (\Lambda_\Tt)$, then $\Omega^{m+1}\tau N =N$.

\smallskip

The converse affirmation can be proved easily. If $N = \Omega^{m+1} \tau N$, and $N\neq 0$ is not a projective $\Lambda_\Tt$-module, then $N$ is a $m$-th syzygy. By Remark \ref{rema gorenstein} this module is in $\uCMP(\Lambda_\Tt)$.\end{proof}

As a corollary, we obtain the next result that generalize the properties known for cluster-tilted algebras: Proposition \ref{prop KR} and Theorem \ref{teorema ralf}.

\begin{teorema}\label{teo m tipo a} Let $\Lambda$ be a $m$-cluster tilted algebra of type $\Aa$ or $\tAa$. Then,
\begin{enumerate}
\item $\Lambda$ is Gorenstein of dimension $d \leq m$.
\item $N \in \uCMP(\Lambda)$ if and only if $\Omega^{m+1} \tau N = N$.
\end{enumerate}
\end{teorema}

\begin{proof}
Part (1) follows from Lemma \ref{lema1}. Part (2) follows from Theorem \ref{main}. 
\end{proof}

\subsection{On the Gorenstein property}

It is known that Theorem \ref{teo m tipo a} does not hold in general for $d$-CY tilted algebras. In \cite[Section 5.3]{KR} there is an example (due to Iyama) of a $d$-CY tilted algebra that is not Gorenstein. Moreover, a recent preprint \cite{Lad2} shows that all finite dimensional $k$-algebras are $d$-CY tilted for some $d > 2$.

\smallskip

Still, there are results in this subject due to Keller and Reiten \cite[Section 4.6]{KR2}, and Beligiannis \cite[Theorem 6.4]{Be15}. Both results ask $\add T$ to be \emph{corigid} in some degree, that is there exist a non negative integer $u$ such that $\Hom_\Cc (\add T, \add T [-t])=0$ for all $1\leq t\leq u$, to conclude that $\End_\Cc(T)$ is Gorenstein.





The $m$-cluster categories $\Cc^m_Q$ of types $\Aa$ and $\tAa$ are special cases of triangulated $(m+1)$-CY categories and the subcategories $\add T$  are $(m+1)$-cluster tilting subcategories, as we pointed out in Remark \ref{rema10}. In the next example we show that the subcategory $\add T$ might not be corigid, hence this result is independent of the mentioned above. 

\begin{ejemplo}\label{ejemploA} Let $\Cc^2_Q$ be the $2$-cluster category, where $Q$ is of type $\Aa_4$ and $T$ is in Figure \ref{A2}. 

\begin{figure}[h!]
\centering
\def\svgwidth{4in}
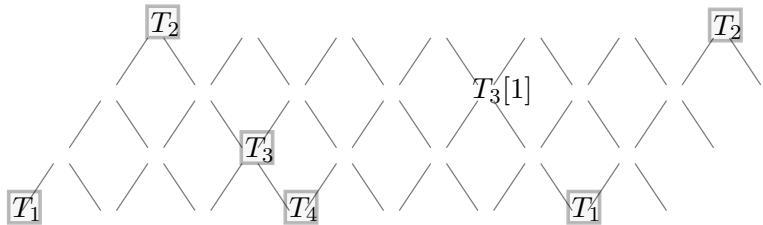
\caption{\sf $2$-cluster tilting object in $\Cc^2_Q$.}
\label{A2}
\end{figure}
\end{ejemplo}

The subcategory $\add T$ in Example \ref{ejemploA} is not corigid since $\Hom(T_3,T_1[-1]) \simeq \Hom (T_3[1],T_1) \neq 0$. By Theorem \ref{teo m tipo a} the $2$-cluster tilted algebra $\End(T)$ is Gorenstein of dimension at most two. In fact, in this example the algebra is of global dimension two. In the next example we see that there are other examples of Dynkin $m$-cluster-tilted algebras where Theorem \ref{teo m tipo a} holds.

\begin{ejemplo}\label{ejemploD} Let $\Cc^2_Q$ be the $2$-cluster category of type $\mathbb{D}_6$, and $T=\oplus_{i=1}^6 T_i$ the $2$-cluster tilting object in Figure \ref{D2}.

\begin{figure}[h]
\centering
\def\svgwidth{5.8in}
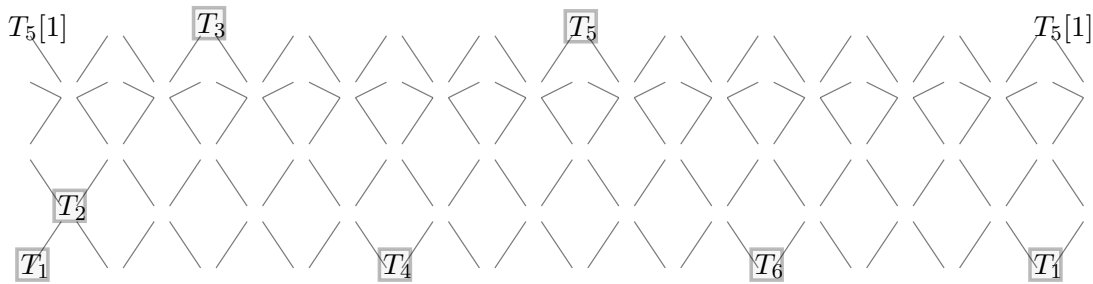
\caption{\sf $2$-cluster tilting object such that $\Hom (T_5, T_3 [-1])=\Hom(T_5[1],T_3) \neq 0$, so $\add T$ is not corigid.}
\label{D2}
\end{figure}
\end{ejemplo}

The algebra $\Lambda=\End_{\Cc^2_Q}(T)$ in Example \ref{ejemploD} is defined by the quiver in Figure \ref{Q ejemplo} modulo the ideal $I= \langle \lambda \alpha, \alpha \beta \gamma, \beta \gamma \delta, \delta \lambda  \rangle$. We show the corresponding AR quiver in Figure \ref{Gamma ejemplo}. We find that $\Lambda$ is Gorenstein of dimension 2 and has infinite global dimension. Also, the modules in $\uCMP(\Lambda)$ are $3,6, \begin{array}{c}
5 \\ 
4
\end{array}, \begin{array}{c}
2 \\ 
1
\end{array}$, exactly those such that $\Omega^{3} \tau N = N$.

\begin{figure}[h!]
\begin{tikzcd}[column sep=.2in, row sep=.18in]
1 && 2 \arrow[ll,"\epsilon"] \arrow[d,"\lambda"] && 3 \arrow[ll,"\delta"] && 4 \arrow[ll, "\gamma"] \\
&& 6 \arrow[rr,"\alpha"] && 5 \arrow[urr,"\beta"']
\end{tikzcd}
\caption{\sf $\Lambda = kQ/I$, Example \ref{ejemploD}}
\label{Q ejemplo}
\end{figure}
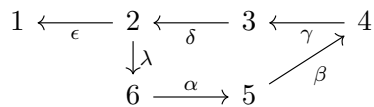
\begin{figure}[h!]

\begin{tikzcd}[column sep=.1in, row sep=.06in]
\vert 1 \arrow[dr] &&  \begin{tabular}{c|}
$2$ \\ 
$6$
\end{tabular} \arrow[dr] \\
& \begin{tabular}{|c}
$2$ \\ 
$16$
\end{tabular} \arrow[dr] \arrow[ur]  && 2 \arrow[dr] && 3 \arrow[dr] && 4 \arrow[dr] && 5 \arrow[dr] && 6 \\
6 \arrow[ur] && \begin{array}{c}
2 \\ 
1
\end{array} \arrow[dr] \arrow[ur] && \begin{array}{c}
3 \\ 
2
\end{array} \arrow[ur] \arrow[dr] && \begin{array}{c}
4 \\ 
3
\end{array} \arrow[dr] \arrow[ur] && \begin{array}{c}
5 \\ 
4
\end{array} \arrow[dr] \arrow[ur] && \begin{array}{c|}
6 \\ 
5
\end{array} \arrow[ur] \\
&&& \begin{tabular}{|c}

$3$ \\ 

$2$ \\ 

$1$ \\ 

\end{tabular} \arrow[ur] \arrow[r]& \begin{array}{|c|}
4 \\ 
3 \\ 
2 \\ 
1
\end{array}  \arrow[r] & \begin{array}{c|}
4 \\ 
3 \\ 
2
\end{array} \arrow[ur] && \begin{array}{|c|}
5 \\ 
4 \\ 
3
\end{array}  \arrow[ur] && \begin{array}{|c|}
6 \\ 
5 \\ 
4
\end{array} \arrow[ur] \end{tikzcd}
\caption{\sf $\Gamma(\mod \Lambda)$, Example \ref{ejemploD}}
\label{Gamma ejemplo}
\end{figure}
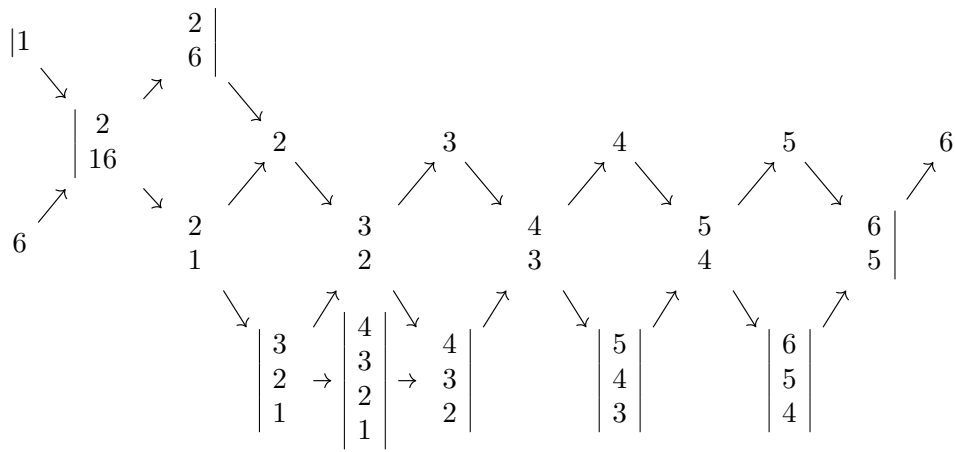

{} 



\end{document}